\documentclass[11pt]{amsart}

\usepackage{amsthm}
\usepackage{amssymb}
\usepackage{graphicx}									 
\usepackage{stmaryrd}                  
\usepackage{footmisc}                  

\usepackage{paralist}
\usepackage{wrapfig}
\usepackage[draft]{pdfcomment}
\usepackage{bbm}

\usepackage[ngerman,russian,british]{babel}
\usepackage[arrow, matrix, curve]{xy}
\usepackage{color}
\usepackage{todonotes}

\usepackage{pict2e}

\usepackage{enumitem}

\newtheorem{thm}{Theorem}[section]

\newtheorem{lem}[thm]{Lemma}

\newtheorem{exe}[thm]{Exercise}
\newtheorem{cor}[thm]{Corollary}
\newtheorem{prp}[thm]{Proposition}
\newtheorem{cnj}[thm]{Conjecture}

\newtheorem{maintheorem}{Theorem}

\theoremstyle{definition}
\newtheorem{dfn}[thm]{Definition}

\newtheorem{qst}[thm]{Question}

\theoremstyle{remark}
\newtheorem{rem}[thm]{Remark}
\newtheorem{exl}{Example}

\newcommand{\R}{\mathbb{R}}

\newcommand{\N}{\mathbb{N}}
\newcommand{\Z}{\mathbb{Z}}

\newcommand{\K}{K}

\newcommand{\Alex}{\mathrm{Alex}}

\newcommand{\To}{\rightarrow}

\newcommand{\Orb}{\mathcal{O}}
\newcommand{\Or}{\mathrm{O}}

\title[Continuous billiard and quasigeodsic flows]{On continuous billiard and quasigeodesic flows \\ characterizing \\  alcoves and isosceles tetrahedra}

\author[C. Lange]{Christian Lange}
\address{Ludwig-Maximilians-Universit\"at M\"unchen, Mathematisches Institut\newline\indent Theresienstra{\ss}e 39, 80333 München, Germany}
\email{lange@math.lmu.de, clange.math@gmail.com}

\begin{document}


\maketitle

\begin{abstract}
We characterize fundamental domains of affine reflection groups as those polyhedral convex bodies which support a continuous billiard dynamics. We interpret this characterization in the broader context of Alexandrov geometry and prove an analogous characterization for isosceles tetrahedra in terms of continuous quasigeodesic flows. Moreover, we show an optimal regularity result for convex bodies: the billiard dynamics is continuous if the boundary is of class $\mathcal{C}^{2,1}$. In particular, billiard trajectories converge to geodesics on the boundary in this case. Our proof of the latter continuity statement is based on Alexandrov geometry methods that we discuss resp. establish first.
\end{abstract}
 

\section{Introduction}

Billiards are a widely studied subject in dynamics and in mathematics in general with many interesting results and open questions, see e.g. the surveys \cite{Kat02,Ta05,Gut12}. For instance, it is not known if every obtuse triangle admits a periodic billiard trajectory, see e.g. \cite{Sc09}. One cause of difficulty is that billiard trajectories through corners are not well defined. The billiard dynamics usually exhibits discontinuities. Nevertheless, there is a reasonable, though ambiguous, notion of a billiard trajectory with bounces in non-smooth boundary points that makes sense in any dimension, see Section \ref{sub:billiards}. Such billiards have for instance been studied in \cite{BC89,Gh04,BB09}. This generalized notion is also essential in the context of the relationship between certain symplectic capacities and shortest billiard trajectories \cite{Ru22,AY14}.

Another important topic in mathematics are reflection groups.  After their prominent appearance in Lie theory they pervaded branches like algebra, topology and geometry, see e.g. \cite{MT02,Dol08,Da11}. Reflection groups are tied to billiards via the reflection law. To each polyhedral billiard table one can associate a group generated by the reflections at the table's faces, which encodes interesting properties about the billiard. The case when the linear part of this group is discrete is of special interest in the context of Teichmüller theory \cite{MT02}. Here we show that the group itself is discrete if and only if the billiard dynamics is sufficiently continuous.

\begin{maintheorem}\label{thm:billiard_orbifold} A polyhedral convex body in $\R^n$ admits a continuous billiard evolution if and only if it is an alcove, i.e. the fundamental domain of a discrete affine reflection group. In particular,  irreducible such billiard tables are classified by connected (affine) Coxeter--Dynkin diagrams.
\end{maintheorem}

Roughly speaking, we say that a convex body admits a continuous billiard evolution if there exists a global choice of billiard trajectories that are defined for all times so that convergence of initial conditions implies pointwise convergence of the trajectories, see Section \ref{sub:billiards} for more details. A convex body in $\R^n$ is an alcove if and only if it is an orbifold, see Proposition \ref{prp:reflection}, i.e. a metric space that is locally isometric to certain model spaces, see Section \ref{sub:orbifold}. In this case the continuous billiard evolution is given by the orbifold geodesic flow which in turn is induced by the geodesic flow of $\R^n$. For more background about orbifold geodesics we refer to e.g. \cite{La18,La19}. An interpretation of Theorem \ref{thm:billiard_orbifold} in the context of Alexandrov geometry will be given further below and in Section \ref{sec:continuous_billiard_quasigeodesic}.

In the class of general convex billiard bodies the situation is much more flexible as the following result illustrates. 

\begin{maintheorem}\label{thm:convex_smooth} Let $K$ be a convex body in $\R^n$ whose boundary is of class $\mathcal{C}^{2,1}$ and has a positive definite second fundamental form. Then $K$ admits a continuous billiard evolution. In particular, billiard trajectories whose initial directions converge to a tangent vector of the boundary converge locally uniformly to the corresponding geodesic of the boundary.
\end{maintheorem}

In particular, Theorem~\ref{thm:convex_smooth} generalizes results by Halpern in dimension $2$ \cite{Hal77} and by Gruber in all dimensions \cite{Gr90} to $\mathcal{C}^{2,1}$ submanifolds, see Lemma \ref{lem:convergence_to_boundary}. In fact, Theorem~\ref{thm:convex_smooth} is optimal in the sense that its statement fails for $\mathcal{C}^{2,\alpha}$ submanifolds for any $\alpha<1$, see \cite{Hal77}. On the other hand, there are also examples of convex bodies in $\R^2$ which admit a continuous billiard evolution, but whose boundary is only of class $\mathcal{C}^1$, see Example \ref{exe:c1_example}.

Nevertheless, some rigidity remains, at least locally. Namely, in dimension $2$ the tangent cone of each point of a convex body that admits a continuous billiard evolution is an orbifold, see Proposition \ref{prp:local_rigidity_dim2}. We suspect that the same conclusion also holds in higher dimensions.

\begin{cnj}\label{conj:orb_point} If a convex body in $\R^n$ admits a continuous billiard evolution (resp. quasigeodesic flow, see below), then all its tangent cones are orbifolds.
\end{cnj}

Theorems \ref{thm:billiard_orbifold} and \ref{thm:convex_smooth} admit the following interpretation in the context of Alexandrov geometry. A convex body is an example of an Alexandrov space with nonnegative curvature, and a (generalized) billiard trajectory corresponds to a so-called quasigeodesic on this Alexandrov space. In this context the question about the existence of a continuous billiard evolution naturally generalizes to the question about the existence of an everywhere defined quasigeodesic flow that satisfies a certain continuity condition, see Section \ref{sub:continuity_alex_sense}. A priori, the geodesic flow of an Alexandrov space with empty boundary is defined almost everywhere for all times and continuous on its domain in many cases \cite{KLP21,BMS22}.

In fact, also our proof of Theorem \ref{thm:convex_smooth} relies on methods from Alexandrov geometry. Namely, we apply a result by Alexander and Bishop about the precise convexity of the distance function to the boundary, see Lemma \ref{lem:convergence_to_boundary}. Moreover, we show and apply the statement that quasigeodesics of a convex body with sufficiently regular boundary that are contained in the boundary are also quasigeodesics of the boundary with respect to its intrinsic metric, see Lemma \ref{lem:quasigeodesic_boundary}.

Another large class of Alexandrov spaces, including those in Theorem \ref{thm:billiard_orbifold}, that admit continuous quasigeodesic flows are quotients of Riemannian manifolds by proper and isometric Lie group actions \cite{LT10}. Like the billiard example, this already indicates that the question of which Alexandrov spaces admit a continuous quasigeodesic flow is complicated in general. Nevertheless, we can say something in the class of polyhedral Alexandrov spaces given by boundaries of convex polyhedral bodies. In this case the same proof as the one of Theorem \ref{thm:billiard_orbifold} shows that a continuous quasigeodesic flow exists if and only if the space is a (flat) orbifold. The following (semi-)rigidity result classifies such spaces. Here a tetrahedron is called \emph{isosceles} if its opposite sides have equal length. Such a tetrahedron is also known as a \emph{disphenoid}.

\begin{maintheorem} \label{thm_3_4_polyhedral} The boundary of a polyhedral convex body in $\R^n$ admits a continuous quasi\-geodesic flow if and only if it is a Riemannian orbifold with respect to its intrinsic metric. The only polyhedral convex bodies that are bounded by Riemannian orbifolds are isosceles tetrahedra in $\R^3$.
\end{maintheorem}

Theorem \ref{thm_3_4_polyhedral} adds to a large number of interesting properties and characterizations of isosceles tetrahedra, see e.g. \cite{AP18,Br26,FF07}. 

The proof of the only if statement of the first part of Theorem \ref{thm_3_4_polyhedral} works like the only if part of Theorem \ref{thm:billiard_orbifold} by induction on the dimension, see Section \ref{sec:polyhedral_billiard_table}. The $3$-dimensional case of the second part can for instance be obtained via the Gauß-Bonnet theorem or with Euler's polyhedral formula, see Section \ref{sub:_prove_C_dim3}. Examples in higher dimensions will be ruled out by a comparison of singular strata with respect to the orbifold structure and the polyhedral structure, see Section \ref{sub:proof_C_dimh}.

Finally, we point out that the everywhere defined continuous quasigeodesic flows in Theorems \ref{thm:billiard_orbifold}, \ref{thm:convex_smooth} and \ref{thm_3_4_polyhedral} are unique, see Section \ref{sub:further_properties}.

\subsection{Structure of the paper} In Section \ref{sec:preliminaries} we collect some preliminaries that are needed later in the paper. Theorem \ref{thm:billiard_orbifold} about continuous billiards on polyhedral convex bodies is then proved in Section \ref{sec:polyhedral_billiard_table}. To read it the subsections of Section \ref{sec:preliminaries} about Alexandrov spaces and quasigeodesics can be skipped. The latter are required in Section \ref{sec:continuous_billiard_quasigeodesic} where we generalize the discussion to arbitrary convex bodies and Alexandrov spaces. Some of these considerations are then applied in Section \ref{sec:general_convex} in the proof of Theorem \ref{thm:convex_smooth}. Finally, in Section \ref{sec:boundary_polyhedral} we prove Theorem \ref{thm_3_4_polyhedral}. While the formulation of Theorem \ref{thm_3_4_polyhedral} relies on the notion of a quasigeodesic, Section \ref{sec:boundary_polyhedral} can be read independently from the sections about Alexandrov spaces and quasigeodesics, either by taking the first part of Theorem \ref{thm_3_4_polyhedral} for granted or by taking the characterization of quasigeodesics in Lemma \ref{cor:quasigeodesic_polyhedral} as a definition.
\newline
\newline
\textbf{Acknowledgements.} This work came into being while the author was visiting ENS de Lyon and Ruhr-Universität Bochum. He thanks the geometry and dynamics groups there for their hospitality. He also would like to thank Alexander Lytchak and Artem Nepechiy for discussions about Alexandrov spaces and quasigeodesics. Moreover, he is grateful to Luca Asselle, Ruth Kellerhals, Florian Lange, Bernhard Leeb, Alexander Lytchak, Anton Petrunin, Daniel Rudolf and Clemens Sämann for useful comments, hints to the literature or answering some question. Finally, he is grateful to the anonymous referee whose remarks helped to improve the exposition and the statement of Theorem B.

\section{Preliminaries}\label{sec:preliminaries}

\subsection{Billiards on polyhedral convex bodies} \label{sub:billiards}
As a warm-up we consider billiards on (polyhedral) convex bodies. At a boundary point with a unique tangent space, a billiard trajectory is reflected according to the usual reflection law, i.e. the angle of incidence equals the angle of reflection. We would like to have a reasonable notion of billiard trajectories that may pass through corners of the boundary of the table. One criterion should be that billiard trajectories are closed under pointwise limits. 

We first introduce the following notions related to a convex body $K\subset \R^n$. The \emph{tangent cone} of $K$ at a point $p$ in $K$ is defined to be
\[
			T_p K = \overline{\left\langle q-p \mid q \in K \right\rangle_{\R} },
\]
i.e. the closure of the $\R$-span of all $q-p$, $q\in \K$. The \emph{normal cone} at $p$ is defined to be
\[
			N_p K = \{ v \in \R^n \mid \left\langle v, u \right\rangle \leq 0 \text{ for all } u \in T_p K \}.
\]
A point in $K$ lies in the interior of $K$ if and only if $N_p K$ consists of a single point. A boundary point for which $N_p K$ is $1$-dimensional is called \emph{smooth}. We call $K$ smooth if all its boundary points are smooth. Two vectors $u,v \in T_p K$ are called \emph{polar} if $-(u+v) \in N_p K$. If $p$ is a smooth boundary point then for any unit vector $v\in T_p K$ there exists a unique polar unit vector $u\in T_p K$ and this correspondence specifies the reflection law at $p$. Let us record the following characterization.

\begin{lem}\label{lem:polar_characterization}
The following conditions are equivalent for two unit vectors $u,v \in T_p K$.
\begin{compactenum}
\item $u$ and $v$ are polar, i.e. $\left\langle u,w\right\rangle + \left\langle v,w\right\rangle\geq 0$ for all $w\in T_p K$.
\item There exists a supporting hyperplane of $K$ at $p$ orthogonal to $u+v$.
\item $\angle (u,w) + \angle (v,w) \leq \pi$ for all $w \in T_p K$, $w \neq 0$.
\end{compactenum}
\end{lem}
Here $\angle (u,w)$ denotes the angle between $u$ and $w$. For a path $c:I \To K$ we denote by $c^+(t_0)$ the right derivative of $c$ at $t_0$ and by $c^-(t_0)$ the right derivative of $t \mapsto c(t_0-t)$ at $0$ if they exist. For now we will restrict ourselves to the case of a \emph{polyhedral} convex body $K$  \cite{Al05}. It will be more natural to consider the general case in the context of Alexandrov geometry, see Section \ref{sec:continuous_billiard_quasigeodesic}.

\begin{dfn} \label{dfn:billiard_trajectory} Let $K$ be a polyhedral convex body. A continuous path $c:\R\supset I \To K$ parametrized proportional to arclength is called \emph{billiard trajectory}, if it is locally length minimizing except at a discrete set of times $\mathcal{T}\subset I$ such that for each $t \in \mathcal{T}$ the vectors $c^+(t)$ and $c^-(t)$ are polar.
\end{dfn}

In particular, each constant path is a billiard trajectory, and in the $2$-dimensional case a parametrization of the boundary of $K$ is a billiard trajectory. Moreover, Lemma \ref{lem:polar_characterization}, $(i)$ implies that pointwise limits of billiard trajectories are indeed again billiard trajectories.

We say that a billiard trajectory $c$ \emph{bounces} at time $t$ if $c^+(t) \neq -c^-(t)$. Billiard  trajectories on polytopes are tame in the following sense.

\begin{lem} \label{lem:collision_estimate} On a polyhedral convex body $K$ bounce times do not accumulate.
\end{lem}
\begin{proof} 
Suppose the bounce times of a billiard trajectory $c:[0,t_0) \To K$ on a polyhedral convex body $K$ accumulate at time $t_0$ and let $p$ be the limit of $c(t)$ as $t$ tends to $t_0$. We can assume that $K=T_pK$. Instead of following the billiard trajectory, we can follow a straight line and reflect the table at a bounce time at the respective supporting hyperplane, see Lemma \ref{lem:polar_characterization}, $(ii)$. Since these reflections fix the point $p$, the only possibility that $c$ runs into $p$ is that the straight line passes through $p$. However, in this case the billiard trajectory experiences only a single bounce near $p$. This contradiction completes the proof of the lemma.
\end{proof}

Alternatively, the statement can be deduced  from \cite{Sin78} which provides a constant $C$ such that any regular billiard trajectory on a tangent cone of a polyhedral convex body experiences at most $C$ bounces. Moreover, in a similar way one can deduce from \cite[Corollary~1]{BFK98} that there exists a constant $C$ only depending on $K$ such that any unit speed billiard trajectory experiences at most $C(t+1)$ bounces in any time interval of length $t$. In particular, each billiard trajectory can be extended for all times. While the latter is still true in more general situations, see Section \ref{sub:quasigeodesics}, Lemma \ref{lem:collision_estimate} may fail on general convex (even smooth) tables, see Section \ref{sub:optimality} and \cite{Hal77}.  It will be more natural to consider such cases in the context of Alexandrov geometry, see Section \ref{sub:Alex_spaces}, as already pointed out.

The tangent cone bundle $TK$ of $K$ is defined to be the union of all tangent cones $T_pK$ of $K$ and it inherits a subspace topology from $T\R^n$. By a \emph{billiard flow} we mean a dynamical system $\Phi: TK \times \R \To TK$, i.e. a map with $\Phi(\cdot,0)=\mathrm{id}_{TK}$ and $\Phi(\Phi(v,s),t)=\Phi(v,s+t)$ for all $v\in TK$ and all $s,t \in \R$, such that for each $v \in TK$ the map $ \R \ni t \mapsto \pi(\Phi(v,t)) \in K$ is a billiard trajectory with initial conditions $\Phi(v,t)$ at time $t$, where $\pi: TK \To K$ is the natural projection. We say that $K$ admits a \emph{continuous billiard evolution}, if there exists a billiard flow for $K$ such that the composition $\pi\circ \Phi$ is continuous. We also call $K$ continuous, if it admits a continuous billiard evolution. Observe that two convex bodies $K_1\subset \R^n$ and $K_2\subset \R^m$ are continuous billiard tables if and only if $K_1 \times K_2 \subset \R^{m+n}$ is so. Examples of polyhedral continuous billiard tables will be constructed in Section \ref{sec:polyhedral_billiard_table}.

\subsection{Riemannian orbifolds}\label{sub:orbifold} An \emph{$n$-dimensional Riemannian orbifold} is a metric length space $\Orb$ such that each point in $\Orb$ has a neighborhood that is isometric to the quotient of an $n$-dimensional Riemannian manifold $M$ by an isometric action of a finite group $\Gamma$ \cite{La20}. For a point $p$ in a Riemannian orbifold $\Orb$ the isotropy group of a preimage of $p$ in a Riemannian manifold chart is uniquely determined up to conjugation. Its conjugacy class in $\Or(n)$ is called the \emph{local group} of $\Orb$ at $p$ and we also denote it as $\Gamma_p$. The point $p$ is called \emph{regular} if this group is trivial and \emph{singular} otherwise. More precisely, an orbifold admits a stratification into manifolds, where the stratum of codimension $k$ is given by
\[
			\Sigma_k = \{ p\in \Orb \mid \mathrm{codim} \mathrm{Fix}(\Gamma_p)=k \}.
\]
In particular, $\Sigma_0$ is the set of regular points. 

Examples of Riemannian orbifolds arise as quotients of Riemannian manifolds by isometric and proper actions of discrete groups. Riemannian orbifolds that can be obtained in this way are called \emph{good} or \emph{developable}. The quotient map from the manifold to the orbifold is then an instance of a Riemannian orbifold covering, cf. e.g. \cite{La20}. It can be useful to have criteria for an orbifold to be good, cf. \cite{LR21}. For instance, if a complete Riemannian orbifold has constant curvature, meaning that all local manifold charts have constant curvature, then it is good \cite{MM91}. This fact has for instance been applied in \cite{Le15} and it also enters the proof of Theorems \ref{thm:billiard_orbifold} and \ref{thm_3_4_polyhedral}.

\subsection{Affine reflection groups}\label{sub:affine_reflection} An \emph{affine reflection group} is a discrete subgroup of the isometry group of a Euclidean vector space $\R^n$ that is generated by reflections. We also assume that it acts cocompactly. Affine reflection groups have been classified by Coxeter \cite{Co34}. Their classification can be conveniently stated in terms of (affine) \emph{Cartan-Dynkin diagrams}. A fundamental domain of an affine reflection group $\Gamma$  acting on $\R^n$, a so-called \emph{alcove}, is given by
\begin{equation*} 
\begin{split}
\Lambda &= \{ q \in \R^n \mid d(p,q)\leq d(p,gq) \text{ for all } g \in \Gamma \} \\
		&= \bigcap_{g\in \Gamma, \; \mathrm{codim}(\mathrm{Fix}(g))=1} \{ q \in \R^n \mid d(p,q)\leq d(p,gq) \},
\end{split}
\end{equation*}
where $p$ is a point in $\R^n$ that is not fixed by any $g\in \Gamma$ \cite[Theorem~4.9]{Hum90}. In particular, it is a polyhedral convex body and isometric to the quotient metric space $\R^n/\Gamma$, which is an orbifold, cf. \cite{Da11}. The converse also holds.

\begin{prp} \label{prp:reflection}  A polyhedral convex body $K$ in $\R^n$ is an alcove if and only if it is an orbifold.
\end{prp}
\begin{proof} A polyhedral convex body $K$ in $\R^n$ which is an orbifold has to be \emph{dihedral} in the sense that all angles between intersecting codimension one faces are integral submultiplies of $\pi$, and any such polytope is an alcove \cite[Theorem~6.4.3 and Proposition~6.3.9]{Da08}
\end{proof}

\subsection{Metric constructions} \label{sub:metric_constructions}
For a metric space $(X,d)$ with $\mathrm{diam}(X)\leq \pi$ a metric $d_c$ on the open cone $CX:=(X\times [0,\infty)) / \sim$, where $\sim$ collapses $X\times \{0\}$ to a point, can be defined as follows, cf. \cite[Def.~3.6.12., Prop.~3.6.13]{MR1835418}. For $q,p \in CX$ with $p=(x,t)$ and $q=(y,s)$ set
\[
			d_c(p,q)=\sqrt{t^2+s^2 - 2ts \cos(d(x,y))}.
\]
The space $(CX,d_c)$ is referred to as the \emph{Euclidean cone} of $(X,d)$. If $X$ is the unit sphere in $\R^n$ with its induced length metric, then $(CX,d_c)$ is naturally isometric to $\R^n$. The assumption $\mathrm{diam}(X)\leq \pi$ of this construction is in particular satisfied for Alexandrov spaces (see Section \ref{sub:Alex_spaces}) with curvature $\geq 1$ \cite[Thm.~10.4.1]{MR1835418}. An isometric action of a group $\Gamma$ on $X$ induces an isometric action of $\Gamma$ on $CX$ in the obvious way and the metric spaces $CX/\Gamma$ and $C(X/\Gamma)$ are isometric.

For a metric space $X$ with a closed subspace $Y$ the natural metric on the double of $X$ along $Y$ is for instance described in \cite[Section~4]{La20} and the references therein.

\subsection{Alexandrov spaces}\label{sub:Alex_spaces} The following two subsections might be skipped on first reading. 

An \emph{Alexandrov space} with curvature bounded below by $\kappa$ is, roughly speaking, a complete, locally compact, geodesic metric space in which triangles are not thinner than their comparison triangles in the model plane of constant curvature $\kappa$, see e.g. \cite{BGP92,MR1835418,AKP19} for more details. The Hausdorff dimension of such a space is always an integer or infinite. In the following we always assume that the dimension is finite. We denote the class of $n$-dimensional Alexandrov spaces with curvature bounded below by $\kappa$ equipped with the $n$-dimensional Hausdorff measure by $\Alex^n(\kappa)$.

Examples of Alexandrov spaces with nonnegative curvature are given by boundaries of convex bodies and by metric doubles of convex bodies along their boundary, cf. \cite[Thm.~10.2.6]{MR1835418}. Another class of examples of Alexandrov spaces is given by quotients of compact Riemannian manifolds by isometric actions of compact Lie groups. In particular, compact Riemannian orbifolds are Alexandrov spaces.

For any pair of points $x$ and $y$ in an Alexandrov space $X$ there exists by definition a geodesic, i.e. a distance realizing path, that connects $x$ with $y$. We denote such a geodesic by $[xy]$, although it is in general not unique. Consider
\[
	\Sigma_x' := \{[xy] | y \in X \backslash \{y\} \}/ \sim.
\]
where the equivalence relation is defined such that $[xy] \sim [xz]$ if and only if $[xy] \subset [xz]$ 
or $[xz] \subset [xy]$. We point out that geodesics in $X$ cannot branch. Measurements of angles defines a metric on $\Sigma_x'$ \cite[§4.3]{MR1835418}. More precisely, the angle between two geodesics $\gamma, \gamma' : [0,\varepsilon) \To X$ starting at $x$ is defined to be
\[
			\limsup_{t,s \downarrow 0} \angle_p (\gamma(s),\gamma'(t)).
\]
Here $\angle_p(q,q')$ denotes the comparison angle $\tilde\angle_{\kappa}(|pq|,|qq'|,|pq'|)$ of a comparison triangle with side lengths $|pq|$, $|qq'|$ and $|pq'|$ in the model plane of constant curvature $\kappa$ opposite to the comparison side of $qq'$. The \emph{space of directions} $\Sigma_x$ of $X$ at $x$ is defined to be the metric completion
 of $\Sigma_x'$. The Euclidean cone over the space of directions is called the \emph{tangent cone} of $X$ at $x$ and is denoted as $T_x X$. Alternatively, the tangent cone $T_xX$ can be obtained as pointed Gromov-Hausdorff limit of rescaled 
versions of $X$ \cite[Theorem 7.8.1]{BGP92}. If  $X \in \Alex^{n}(\kappa)$ for some $\kappa$, then $\Sigma_x \in \Alex^{n-1}(1)$ and $T_x X\in \Alex^{n-1}(0)$.
 
The boundary of an Alexandrov space can be defined inductively via the spaces of directions. A point belongs to the boundary if and only if the boundary of its space of directions is non-empty. The interior is the complement of the boundary. 
The two classes of examples arising from convex bodies mentioned above have empty boundary. 

\subsection{Quasigeodesics}\label{sub:quasigeodesics}

Geodesics in Riemannian manifolds can be characterized among curves parametrized by arclength in terms of a certain concavity condition \cite{QG95}. A curve in an Alexandrov space parametrized by arclength that satisfies this condition does not need to be a geodesic and is called a quasigeodesic. Here we recall a definition of quasigeodesics in terms of so-called developments from \cite{Pe07,QG95}, see also \cite[Section II.8.E]{AKP19}.

Consider a curve $\gamma:[a,b] \To X$ parametrized by arclength in an Alexandrov space $X\in \Alex^n(\kappa)$. We pick a point $p\in X\backslash \gamma$ and assume that $0<|p\gamma(t)|<\pi/ \sqrt{k}$ for all $t \in [a,b]$ if $\kappa>0$. Then, given a reference point $o$, up to rotation there exists a unique curve $\tilde\gamma :[a,b] \To S_{\kappa}$ parametrized by arclength in the model plane $S_{\kappa}$ of constant curvature $\kappa$ such that $|o\tilde \gamma(t)| = |p \gamma(t)|$ for all $t$ and the segment $o\tilde \gamma(t)$ turns clockwise as $t$ increases. The curve $\tilde\gamma$ is called the \emph{development} of $\gamma$.

We call a curve $\gamma:[a,b] \To X$ a (local) \emph{quasigeodesic} if for any $t_0 \in [a,b]$ there exists a neighborhood $U$ of $\gamma(t_0)$ and an $\varepsilon>0$ with $\gamma([t_0-\varepsilon,t_0+\varepsilon])\subset U$ such that the development $\tilde\gamma$ of the restriction $\gamma_{|[t_0-\varepsilon,t_0+\varepsilon]}$  with respect to any point in $U$ is convex in the sense that for every $t\in (t_0-\varepsilon,t_0+\varepsilon)$ and for every $\tau>0$ the region bounded by the segments $o\tilde\gamma(t \pm \tau)$ and the arc $\tilde\gamma_{|[t-\tau,t+\tau]}$ is convex whenever it is defined.

Quasigeodesics have nice properties. For instance, they are unit speed curves and have uniquely defined left and right tangent vectors \cite[Section~5.1]{Pe07}. Here the right tangent vector $\gamma^+(0)$ of a quasigeodesic $\gamma$ is defined to be the limit in $\Sigma_x$ of the directions $[x\gamma(t)]$ for $t\searrow 0$, cf. \cite[Thm.~A.0.1]{Pe07}. Moreover, for any point $x \in X$ and any direction $\xi \in \Sigma_x$ there exists a quasigeodesic with $\gamma(0)=x$, $\gamma^+(0)=\xi$ \cite[Section~5.1]{Pe07}. Left- and right derivatives $\gamma^-(t_0)$ and $\gamma^+(t_0)$ of a quasigeodesic $\gamma$ are polar, i.e. $\angle (\gamma^-(t_0),w)+ \angle ( \gamma^+(t_0),w) \leq \pi$ for any $w\in \Sigma_p$ \cite[Section~2.2]{QG95}. Conversely, if $\gamma_1:(s_0,0] \To X$ and $\gamma_2:[0,t_0) \To X$ are quasigeodesics such that $\gamma_1(0)=\gamma_2(0)$ and such that $\gamma_1^-(0)$ and $\gamma_1^+(0)$ are polar, then also the concatenation of $\gamma_1$ and $\gamma_2$ is a quasigeodesic. Moreover, pointwise limits of quasigeodesics are quasigeodesics. Let us also record the nontrivial statement that quasigeodesics can be extended for all times \cite{QG95,Pe07}.

Finally, we repeat the the following two statements from \cite[Section~2.3]{QG95}.

\begin{lem}\label{lem:unique_geodesic} If a geodesic starts in a given direction in an Alexandrov space $X$, then any quasigeodesic with the same initial direction coincides with it for some positive time.
\end{lem}

\begin{cor}\label{cor:quasigeodesic_Riemannian} A quasigeodesic in a Riemannian manifold $M$ is a geodesic.
\end{cor}

\section{Polyhedral billiard tables} \label{sec:polyhedral_billiard_table}

Let us first show the if direction of Theorem \ref{thm:billiard_orbifold}. 

\begin{prp} \label{prp:polyhedral_continuous} A polyhedral convex body $K$ in $\R^n$, which is a Riemannian orbifold, admits a continuous billiard evolution.
\end{prp}
\begin{proof} 
By Proposition \ref{prp:reflection} we can realize $K$ as the quotient of $\R^n$ by an affine reflection group $\Gamma$. We claim that the geodesic flow on $\R^n$ induces a billiard flow on $\Lambda$. For that consider a geodesic $c$ in $\R^n$. If at some time $t$ it does not intersect the fixed point subspace of an element in $\Gamma$ transversely, then the image of this geodesic in the quotient is locally length minimizing at time $t$. Since $\Gamma$ is discrete, the set of times where this condition is not satisfied is discrete. Since the induced quotient maps $ T^1_p \R^n \To T^1 (\R^n/\Gamma_p)=(T^1 \R^n)/\Gamma_p$ are $1$-Lipschitz, it follows that geodesics in $\R^n$ project to billiard trajectories in $\Lambda$, see Lemma \ref{lem:polar_characterization}, $(ii)$. Moreover, since the geodesic flow on $\R^n$ is a continuous dynamical system, the so induced map is indeed a billiard flow which defines a continuous billiard evolution.
\end{proof}

The billiard evolutions induced by this construction on equilateral triangles and on rectangles in $\R^2$ are illustrated in Figure \ref{fig:2d_continuous_billiards}.

\begin{figure}
	\centering
		\def\svgwidth{0.7\textwidth}
		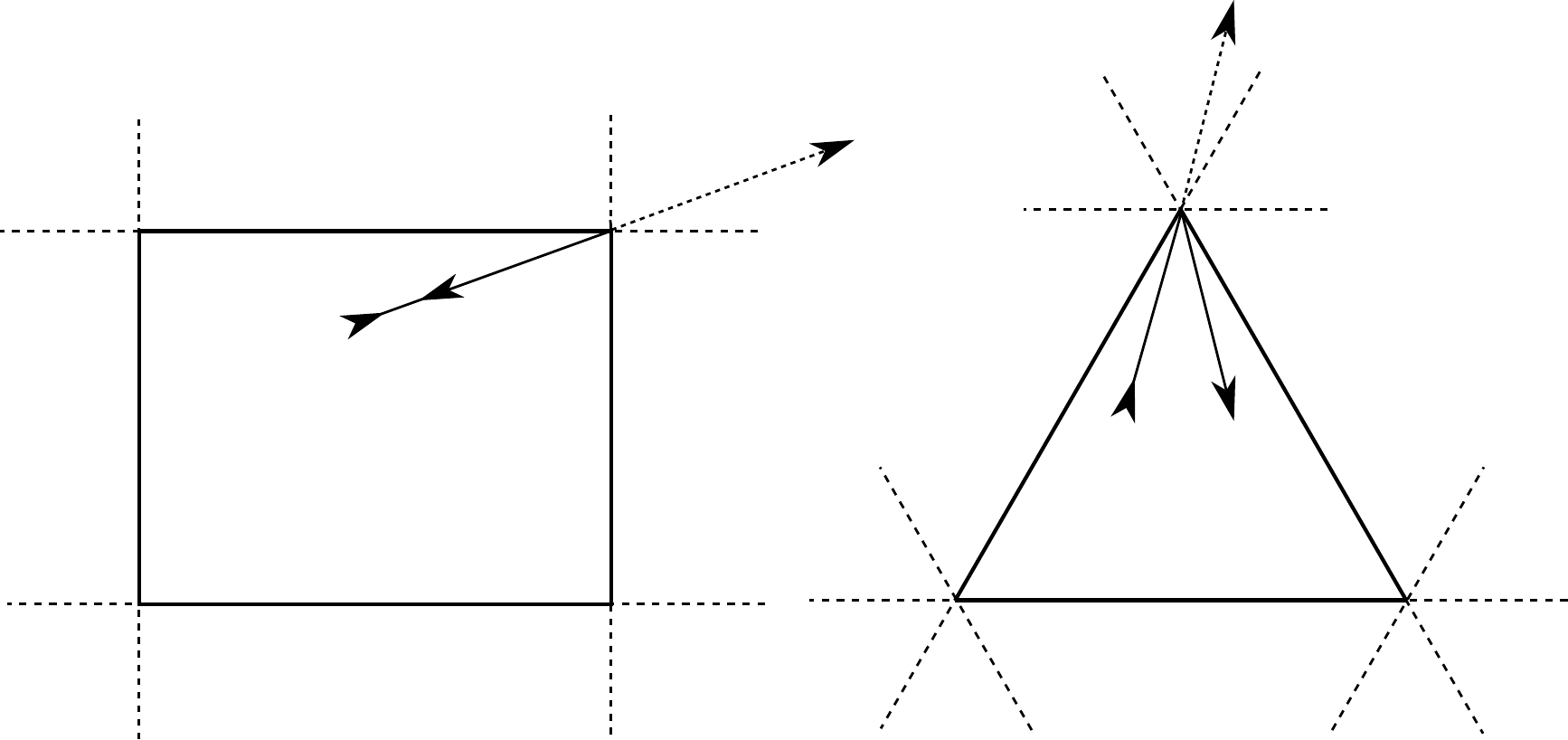
	\caption{Continuous billiard tables in $\R^2$ of type $A_1 \times A_1$ and $A_2$.}
	\label{fig:2d_continuous_billiards}
\end{figure}

Now we will prove the only if direction of Theorem \ref{thm:billiard_orbifold}. We start with a proof in dimension~$2$.

\begin{figure}
	\centering
		\def\svgwidth{0.9\textwidth}
		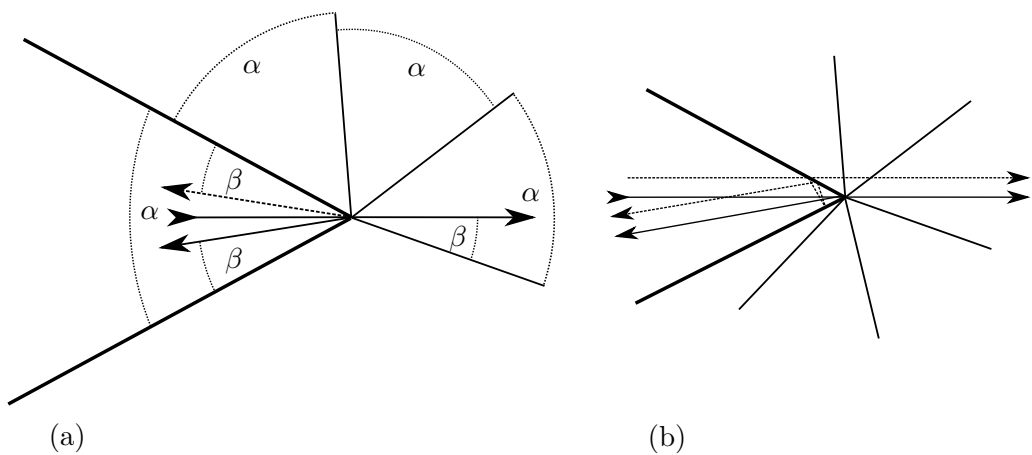
	\caption{Reflection at a corner with opening angle $\frac{\pi}{4}<\frac{2\pi}{7}<\alpha <\frac{\pi}{3}$. The boundary of the table is depicted in solid.	On the unfolded table a billiard trajectory that does not hit the corner corresponds to a straight line.	a)  The two reflections that are obtained as limits of billiard trajectories that approximate the bisector parallelly from above (solid) and from below (dashed). b) Approximation of the bisector from above by a parallel trajectory that does not hit the corner.}
	\label{fig:2d_corner}
\end{figure}

\begin{lem} \label{lem:dim2_billiard} A polyhedral convex body $K$ in $\R^2$ that admits a continuous billiard evolution is an orbifold. The only possible billiard table shapes are rectangles and triangles with interior angles $(\pi/3,\pi/3,\pi/3)$, $(\pi/2,\pi/4,\pi/4)$ and $(\pi/2,\pi/3,\pi/6)$ corresponding to the affine reflection groups of type $A_1\times A_1$, $A_2$, $BC_2$ and $G_2$ (see Figure \ref{fig:2d_continuous_billiards}).
\end{lem}
\begin{proof} Let $K$ be a polyhedral convex body as in the statement of the lemma. It is sufficient to show that $K$ is an orbifold in a neighborhood of each corner. Suppose there is a corner where this is not the case. The opening angle $\alpha$ of this corner then satisfies $\frac{\pi}{n+1} < \alpha < \frac{\pi}{n}$ for some $n\in \N_{\geq 2}$. Figure $\ref{fig:2d_corner}$ illustrates this situation in the case of $n=3$. Consider a billiard trajectory that runs into the corner as a bisector. We approximate this trajectory by parallel trajectories that approach the corner slightly above and below the bisector. The approximation from above is illustrated in Figure $\ref{fig:2d_corner}$, (b). We can understand these billiard trajectories by continuing them as straight lines and reflecting the table at the table's faces instead as shown in Figure $\ref{fig:2d_corner}$. In particular, this shows that these approximating trajectories do not hit the corner and that their continuations are thus uniquely defined.

Let $\beta$ be the angle at the corner between the continuation of the bisector and the first cushion in the development of the table that does not point in the interior of the upper half plane, see Figure $\ref{fig:2d_corner}$, (a). In formulas, $\beta = (\frac 1 2 + m) \alpha - \pi$ where $m$ is the minimal integer such that this expression is nonnegative. Depending on whether $n$ is even or odd, the limit of the billiard trajectories that approximate the bisector from above will then be reflected at the corner and form an angle of $\beta$ with the lower or upper face of the table, respectively, see Figure $\ref{fig:2d_corner}$, (a). 

Because of the $\Z_2$-symmetry with respect to the horizontal bisector, the limits of the two approximating sequences of the bisector are mirror images of each other with respect to the bisector as well. Therefore, the two reflections coincide if and only if the angle $\beta$ satisfies $\beta=\frac \alpha 2$. This implies $\alpha = \frac{\pi}{m}$ in contradiction to our assumption.

The second claim follows from the classification of compact, flat $2$-orbifolds, see e.g. \cite[p.~1024]{MT02} or \cite{Da11}.
\end{proof}

Now we reduce the general case to the $2$-dimensional case.

\begin{prp} \label{prp:dim_high}  A polyhedral convex body $K$ in $\R^n$ that admits a continuous billiard evolution is an orbifold.
\end{prp}
\begin{proof} Let $K$ in $\R^n$, $n\geq 3$, be as in the statement of the proposition and let $p \in K$ be a point that lies in a face $\sigma$ of codimension two. We look at the intersection of $K$ with a $2$-dimensional plane $H$ through $p$ orthogonal to $\sigma$. An application of Lemma \ref{lem:dim2_billiard} to the flow restricted to $K \cap H$ in a neighborhood of $p$ shows that the dihedral angle at $\sigma$ is an integral submultiple of $\pi$. As in the proof of Proposition \ref{prp:reflection} we deduce that $K$ is an orbifold by referring to \cite[Theorem~6.4.3 and Proposition~6.3.9]{Da08}
\end{proof}

Proposition \ref{prp:dim_high} completes the proof of Theorem \ref{thm:billiard_orbifold}. We remark that the statement in \cite[Theorem~6.4.3]{Da08} can be proved by induction on the dimension based on the fact that a complete Riemannian orbifold of constant sectional curvature is developable \cite{MM91}, cf. \cite[Proposition~2.16.]{Le15}.

\section{Continuous quasigeodesic flows} \label{sec:continuous_billiard_quasigeodesic}

\subsection{From billiard to quasigeodesic flows}

By a unit speed \emph{billiard trajectory} on a convex body $K$ we simply mean a quasigeodesic on $K$. This definition is justified by the following lemma.

\begin{lem}On a polyhedral convex body $K$ a curve parametrized by arclength is a billiard trajectory in the sense of Definition \ref{dfn:billiard_trajectory} if and only if it is a quasigeodesic.
\end{lem}
\begin{proof} Recall from Section \ref{sub:quasigeodesics} that the concatenation of two curves parametrized by arclength is a quasigeodesic if and only of both curves are quasigeodesics and the initial direction of the second curve is polar to the initial direction of the reversed first curve. Now the claim follows from the fact that quasigeodesics in a Riemannian manifold, in particular in the interior of $K$, are geodesics \cite[Corollary~2.3]{QG95}, and that bounce times of billiard trajectories and quasigeodesics do not accumulate by Lemma \ref{lem:collision_estimate}.
\end{proof}

More generally, a compact Alexandrov space is called \emph{polyhedral} if it admits a triangulation such that each simplex is globally isometric to a simplex in Euclidean space, cf. \cite{LP15,Le15}. In particular, (boundaries and doubles of) polyhedral convex bodies are examples of polyhedral Alexandrov spaces.

\begin{lem}\label{cor:quasigeodesic_polyhedral}  A continuous curve $c$ in a polyhedral Alexandrov space is a quasigeodesic if and only if it is locally distance realizing except at a discrete number of times $\mathcal{T}$ such that for each $t \in \mathcal{T}$ the vectors $c^-(t)$ and $c^+(t)$ are polar. 
\end{lem}
\begin{proof} If the image under a quasigeodesic $c$ of a small neighborhood around some time $t$ is contained in a simplex, then $c$ is locally length realizing on this interval by the fact that quasigeodesics in Riemannian manifolds are geodesics \cite[Corollary~2.3]{QG95}. The fact that times for which this is not the case do not accumulate follows as in the proof of Lemma \ref{lem:collision_estimate}.
\end{proof}

Given the notion of a billiard trajectory on a convex body, we can define the notions of a billiard flow and of a continuous billiard evolution precisely as in Section \ref{sub:billiards}.

For general Alexandrov spaces we make the analogous definition: By a \emph{quasigeodesic flow} we mean a dynamical system $\Phi: TX \times \R \To TX$, i.e. a map with $\Phi(\cdot,0)=\mathrm{id}_{TX}$ and $\Phi(\Phi(v,s),t)=\Phi(v,s+t)$ for all $v\in TX$ and all $s,t \in \R$, such that for each $v \in TX$ the map $\R \ni t \mapsto \pi(\Phi(v,t)) \in X$ is constant or a constant speed curve with initial conditions $\Phi(v,t)$ at time $t$ which becomes a quasigeodesic with respect to unit-speed parametrization. Here $\pi: TX \To X$ is the natural projection.

\subsection{Continuity in the Alexandrov sense}\label{sub:continuity_alex_sense} In order to be able to talk about continuous quasigeodesic flows we introduce the following convergence notion suggested by Lytchak. We say that the initial directions of a sequence of quasigeodesics $\gamma_i :[0,a)\rightarrow X$, which are uniquely defined unit vectors in $T_{\gamma_i(0)}X$ as discussed in Section \ref{sub:quasigeodesics}, \emph{converge (in the Alexandrov sense)} to the initial direction of a quasigeodesic $\gamma: [0,a)\rightarrow X$ if $\gamma_i(0)\rightarrow \gamma(0)$ and the initial direction of the limit of any convergent subsequence of the $\gamma_i$ (which is again a quasigeodesic) coincides with the initial direction of $\gamma$. In particular, this condition is satisfied if the sequence of quasigeodesics $\gamma_i$ converges pointwise to the quasigeodesic $\gamma$.

In order to compare this convergence notion with others, we prove the following lemma.

\begin{lem}\label{lem:convergence_initial_directions} Let $K$ be a convex subset of a Riemannian manifold. Suppose a sequence of quasigeodesics $\gamma_i:[0,a) \To K$ converges to a quasigeodesic $\gamma:[0,a) \To K$ and that $\gamma^+(0)$ is contained in an $\R$-factor of $T_{\gamma(0)} K$. Then the initial directions $\gamma_i^+(0)$ converge to the initial direction $\gamma^+(0)$ with respect to the subspace topology of $TK$.
\end{lem}
\begin{proof}Set $p=\gamma(0)=\lim_{i\rightarrow \infty} \gamma_i(0)$. Since $\gamma^+(0)$ is contained in an $\R$-factor of $T_{\gamma(0)}K$, we can choose a sequence of points $p_j\in K$ that converge to $p$ from the direction $-\gamma^+(0)$, i.e. initial directions of minimizing geodesics from $p$ to $p_j$ converge to $-\gamma^+(0)$. Now suppose the conclusion of the lemma does not hold. In this case we can assume that the $\gamma^+_i(0)$ converge to another direction $v\neq \gamma^+(0)$ in $T_pK$ with $\angle_p(v,\gamma^+(0))>\varepsilon$ for some $\varepsilon>0$. We can choose a small $t_0>0$ and a large $j$ such that $\angle_p(p_j,\gamma(t_0)) > \pi-\varepsilon/200$. Then for all sufficiently large $i$ we have $\angle_{\gamma_i(0)}(p_j,\gamma_i(t_0)) > \pi-\varepsilon/100$. Moreover, for sufficiently large $i$ and sufficiently small $t_1>0$ we have $\angle_{\gamma_i(0)}(p_j,\gamma_i(t_1)) < \pi- \varepsilon/2$. Hence, the angle $[0,t_0)\ni t \mapsto \angle_{\gamma_i(0)}(p_j,\gamma_i(t))$ is not decreasing in contradiction to the fact that $\gamma_i$ is a quasigeodesic, see \cite[5.(iii), p.~36]{Pe07}. 
\end{proof}

We apply the previous lemma in the proof of

\begin{lem} \label{cor:convergence_sub_alex}Let $K$ be a convex body in $\R^n$ which is smooth or polyhedral. Suppose the initial directions $\gamma_i^+(0)$ of a sequence of quasigeodesics $\gamma_i:[0,a) \To K$ converge to the initial direction $\gamma^+(0)$ of a quasigeodesic $\gamma:[0,a) \To K$ with respect to the subspace topology of $TK$. Then convergence also holds in the Alexandrov sense.
\end{lem}
\begin{proof} Suppose first that $\gamma^+(0)$ points into the interior of $K$. Then there exists some $\varepsilon>0$ such that $\gamma_{|(0,\varepsilon]}$ and $\gamma_{i|(0,\varepsilon]}$ for all sufficiently large $i$ are contained in the interior of $K$. This follows from the observation that the map which sends a unit tangent vector $v\in T_p K$ pointing into the interior of $K$ to the (only) intersection (besides perhaps $p$) of the ray in the direction of this vector with the boundary of $K$ is continuous. Then $\gamma_{|[0,\varepsilon]}$ and $\gamma_{i|[0,\varepsilon]}$ for all sufficiently large $i$ are straight segments, and our assumption implies that $\gamma_{i|[0,\varepsilon]}$ converges pointwise to $\gamma_{|[0,\varepsilon]}$. Hence, in this case the claim holds.

Otherwise, we claim that there exists a convex set $K' \supseteq K$ and some $\varepsilon>0$ such that $\gamma^+(0)$ is contained in an $\R$-factor of $T_p K'$, $p=\gamma(0)$, and that $\gamma_{|[0,\varepsilon]}$ and $\gamma_{i|[0,\varepsilon]}$ for all sufficiently large $i$ are also quasigeodesics of $K'$. In the smooth case we can simply take $K'=K$. In the polyhedral case we first write $K$ as an intersection of finitely many closed half spaces $H_j$, $j=1,\ldots,n$, and define $K'$ to be the intersection of those $H_j$ for which $-\gamma^+(0)$ does not point outside of $H_j$ (at $p$). Then $K' \supseteq K$ and  $\gamma^+(0)$ is contained in an $\R$-factor of $T_p K'$ by construction. We choose $\alpha>0$ such that each unit vector $v\in T_pK$ with $\angle(v,\gamma^+(0))< \alpha$ points into the interior of each discarded closed half plane. Our assumption implies $\angle(\gamma_i^+(0),\gamma^+(0))< \alpha$ for all sufficiently large $i$. We choose $\varepsilon>0$ such that for sufficiently large $i$ the restrictions $\gamma_{|[0,\varepsilon]}$ and $\gamma_{i|[0,\varepsilon]}$ do not bounce at a supporting hyperplane of $K$ that does not contain $p$. Since reflections at supporting hyperplanes containing $p$ and $p+\gamma^+(0)$ leave the condition $\angle(v,\gamma^+(0))< \alpha$ invariant, we see that $\gamma_{|[0,\varepsilon]}$ and $\gamma_{i|[0,\varepsilon]}$ for sufficiently large $i$ do not bounce at the boundary of a discarded half plane. Hence, $\gamma_{|[0,\varepsilon]}$ and $\gamma_{i|[0,\varepsilon]}$ are also quasigeodesics of $K'$.

Now let $\gamma_{i_j}$ be a convergent subsequence of the sequence $\gamma_{i}$ with limit quasigeodesic $\bar \gamma$. Then the initial directions $\gamma^+_{i_j}(0)$ converge to $\gamma^+(0)$ by assumption and to $\bar \gamma^+(0)$ by Lemma \ref{lem:convergence_initial_directions}. Hence, $\bar \gamma^+(0)=\gamma^+(0)$ and so the claim follows.
\end{proof}

Note that the converse of Lemma \ref{cor:convergence_sub_alex} fails at the boundary of $K$. However, an equivalence between different continuity notions for quasigeodesic flows still holds, see Proposition \ref{prp:equivalence_cont_def}.

We say that a quasigeodesic flow gives rise to a \emph{continuous quasigeodesic evolution} if the following condition is satisfied. If for some sequence $(p_i,v_i) \in TX$ and for some $(p,v) \in TX$ the initial directions of the quasigeodesics $[0,\infty) \ni t \mapsto \Phi((p_i,v_i),t)$ converge to the initial direction of the quasigeodesic $[0,\infty) \ni t \mapsto \Phi((p,v),t)$ in the Alexandrov sense, then $\pi(\Phi((p_i,v_i),t_i))$ converges to $\pi(\Phi((p,v),t))$ for all $t$ and all sequences $t_i\rightarrow t$. A quasigeodesic flow gives rise to a continuous quasigeodesic evolution if and only if it is \emph{continuous} in the sense that in the condition above $\Phi((p_i,v_i),t_i))$ converges to $\Phi((p,v),t)$ in the Alexandrov sense for all $t\in [0,\infty)$ and all sequences $t_i\rightarrow t$.

For instance, in the class of examples given by quotients of compact Riemannian manifolds $M$ by isometric actions of compact Lie groups a continuous quasigeodesic flow is induced by the projection of the horizontal geodesic flow on $M$ to the quotient, cf. \cite{LT10}.

Moreover, the quasigeodesic flow of an Alexandrov space in which quasigeodesics do not branch is continuous:

\begin{lem}\label{lem:unique_implies_continuity} If quasigeodesics on an Alexandrov space do not branch, then the uniquely defined quasigeodesic flow is continuous.
\end{lem}
\begin{proof}Suppose to the contrary that there is a sequence of quasigeodesics $\gamma_i :[0,a)\rightarrow X$ whose initial directions converge to the initial direction of a quasigeodesic $\gamma: [0,a)\rightarrow X$ in the Alexandrov sense, but which does not converge pointwise to $\gamma$ respecting the parametrizations. Then there also exists a converging subsequence whose limit is a quasigeodesic distinct from $\gamma$ but with the same initial direction as $\gamma$ in contradiction to our assumption.
\end{proof}

We will apply this statement in the proof of Theorem \ref{thm:convex_smooth}, see Section \ref{sub:proof_thm_B}.

We call a unit tangent vector $v \in TX$ \emph{bifurcating} if it is the initial direction of two quasigeodesics $\gamma_1, \gamma_2:[0,a) \To X$ whose restrictions to $[0,\varepsilon]$ disagree for any $\varepsilon > 0$. In case of a convex body $K$ any tangent vector that points into the interior of $K$ is not bifurcating. Moreover, no tangent vector of a polyhedral convex body is bifurcating, cf. Lemma \ref{cor:quasigeodesic_polyhedral}.

\begin{lem}\label{lem:bifurc_conv} Suppose the initial directions of a sequence of quasigeodesics $\gamma_i:[0,a) \To X$ converge to the initial direction of a quasigeodesic $\gamma:[0,a) \To X$ in the Alexandrov sense. If $\gamma^+(0)$ is not bifurcating, then there exists some $\varepsilon > 0$ such that $\gamma_{i|[0,\varepsilon]}$ converges pointwise to $\gamma_{|[0,\varepsilon]}$.
\end{lem}
\begin{proof} Otherwise there exists for any $\varepsilon>0$ a subsequence $\gamma_{i_j}$ of $\gamma_{i}$ such that $\gamma_{i_j|[0,\varepsilon]}$ converges pointwise to a quasigeodesic distinct from $\gamma_{|[0,\varepsilon]}$ but with the same initial direction. This contradicts the assumption that $\gamma^+(0)$ is not bifurcating.
\end{proof}

In case of a convex body $K$, which is smooth or polyhedral, the following proposition shows that continuity of a quasigeodesic flow is equivalent to continuity of the corresponding billiard evolution in our earlier sense.

\begin{prp} \label{prp:equivalence_cont_def} A quasigeodesic flow on a convex body $K$, which is smooth or polyhedral, is continuous in the Alexandrov sense if and only if the induced billiard evolution is continuous with respect to the subspace topology of $TK$.
\end{prp}
\begin{proof}By Lemma \ref{cor:convergence_sub_alex} a continuous quasigeodesic flow on $K$ induces a continuous billiard evolution on $K$. Conversely, suppose we are given a quasigeodesic resp. billiard flow $\Phi:\R \times TK \To TK$ that induces a continuous billiard evolution. To prove that $\Phi$ is continuous in the Alexandrov sense it suffices to show that if a sequence $v_i\in TK$ converges to $v\in TK$ in the Alexandrov sense (with respect to the quasigeodesics defined by $\Phi$), then there exists a sequence $s_i \searrow 0$ such that $\Phi(s_i,v_i)$ converges to $v\in TK$ with respect to the subspace topology of $TK$.

If $K$ is polyhedral or if $v$ points into the interior of $K$, then by Lemma \ref{lem:bifurc_conv} there is some $\varepsilon>0$ such that the quasigeodesics $\gamma_i:[0,\varepsilon]\To K$ defined by $\Phi$ with initial condition $\Phi(v_i)$ converge pointwise to the quasigeodesic $\gamma:[0,\varepsilon] \To K$ defined by $\Phi$ with initial condition $\Phi(v)$. Here $\varepsilon$ can be chosen such that $\gamma$ is a straight segment. For sufficiently large $i$ and small $\varepsilon$ the number of bounces of $\gamma_i$ is uniformly bounded, see the citation of \cite{BFK98} in Section \ref{sub:billiards} for the polyhedral case. In the case in which $K$ is smooth and $v$ points into the interior of $K$ the bound can be take to be one because of the continuity of outer normal vectors of $K$, cf. the argument in the first paragraph of the proof of Lemma \ref{cor:convergence_sub_alex}. Hence, $\gamma_i$ are polygonal chains with a uniform bound on the number of breaks which converge pointwise to $\gamma$ and which have the same arclength as $\gamma$. This implies the existence of the desired $s_i$ in this case, for, otherwise the chains $\gamma_i$ would not make enough progress in the direction of $\gamma$. 

If $K$ is smooth and $v$ is tangent to its boundary, we can take $s_i=0$. Otherwise there would be a subsequence $\gamma_{i_j}$ of $\gamma_i$ whose initial directions converge to some $w\neq v$ with respect to the subspace topology, and which converges pointwise to a quasigeodesic with initial direction $v$. Since any tangent cone at the boundary of $K$ is a closed half space, 
Lemma \ref{lem:convergence_initial_directions} implies that the initial directions of $\gamma_{i_j}$ converge to $v$ with respect to the subspace toplogy, a contradiction. This completes the proof of the proposition.
\end{proof}

\begin{rem} The statements of Lemma \ref{cor:convergence_sub_alex} and Proposition \ref{prp:equivalence_cont_def} for general convex bodies require an additional argument which is not treated here.
\end{rem}

We close this section with the following question.
\begin{qst}\label{qst:alex_topology} Is there a natural topology on the tangent cone bundle of an Alexandrov space that induces the notion of convergence in the Alexandrov sense introduced above?
\end{qst}
\subsection{Uniqueness of continuous quasigeodesic flows}\label{sub:further_properties}

In this section we discuss uniqueness and related properties of continuous quasigeodesic flows. These discussions are not needed in later sections and could be skipped.

Recall from Lemma \ref{lem:unique_geodesic} that no quasigeodesic can branch from a geodesic. This implies that the continuous quasigeodesic flows in the settings of Theorems \ref{thm:billiard_orbifold} and \ref{thm_3_4_polyhedral}, which are induced by an orbifold geodesic flow, are unique: in a polyhedral Alexandrov space (or an orbifold) a quasigeodesic in the interior of a face (stratum) is a geodesic and can thus only branch, when it hits a face (a singular stratum) of a higher codimension which is at least $2$. Besides, in such spaces any quasigeodesic for which this happens at time $t_0$ can be approximated by quasigeodesics for which this does not happen before time  $t_0+ \varepsilon$ for some $\varepsilon>0$. Therefore, in these cases a continuous quasigeodesic flow is determined by the behaviour of its geodesics. 

Uniqueness of the quasigeodesic flow in the setting of Theorem \ref{thm:convex_smooth} will be shown in Section \ref{sub:proof_thm_B}. In fact, in this case we first show that quasigeodesics cannot branch. In summary, we have

\begin{prp}\label{prp:uniqueness_rds} The continuous quasigeodesic flows in Theorems \ref{thm:billiard_orbifold}, \ref{thm:convex_smooth} and \ref{thm_3_4_polyhedral} are unique.
\end{prp}

We remark that on many Alexandrov space without boundary the geodesic flow exists almost everywhere for all times \cite{BMS22,KLP21,KL20}. By Lemma \ref{lem:unique_geodesic} each quasigeodesic flow coincides with the geodesic flow on the domain of the latter. If each initial direction of a quasigeodesic is a limit in the Alexandrov sense of initial directions of geodesics from an almost everywhere defined geodesic flow, then a continuous quasigeodesic flow is uniquely determined and a reversible dynamical system by continuity if it exists. For instance, this conclusion holds for quotients of compact Riemannian manifolds $M$ by isometric actions of compact Lie groups, cf. \cite[Proposition~12.1]{KL20}. In this case convergence in the Alexandrov sense is equivalent to convergence with respect to the quotient topology of the tangent cone bundle. 

In particular, boundaries and doubles of convex bodies have geodesic flows that exist almost everywhere for all times \cite{KLP21}. Uniqueness of continuous quasigeodesic flows on convex bodies would follow if each continuous quasigeodesic flow could be lifted to a continuous quasigeodesic flow on its double. Note in this respect that quasigeodesics on a convex body are precisely the projections of quasigeodesics of its double. A negative answer to Question \ref{sec:continuous_nonsmooth_bound} in Section \ref{sub:optimality} would imply that such a lift always exists.

\section{Billiards on convex bodies with $\mathcal{C}^{2,1}$ boundary} \label{sec:general_convex}

\subsection{Proof of Theorem \ref{thm:convex_smooth}}\label{sub:proof_thm_B} In this section we prove Theorem \ref{thm:convex_smooth} and thereby illustrate the flexibility of general continuous convex billiard tables. Our proof relies at several places on a Taylor expansion of the exponential map. Recall that on a $\mathcal{C}^{2,1}$ manifold geodesics are of class $\mathcal{C}^{2,1}$, see e.g. \cite[Theorem~14]{dC92,Mi15}. Therefore, for a fixed point $p \in M$ and a fixed unit vector $u\in T_p M$ we have a Taylor expansion of the form 
\begin{equation}\label{eq:exp_taylor}
c_u(s):= \exp_p(su) = p+su+\frac{1}{2} \Pi_p(u,u)s^2 + o(s^2),
\end{equation}
where $\Pi$ denotes the second fundamental form of $M$, cf. \cite{dC92,MMS14}. The latter depends Lipschitz continuously on $p$. 
\begin{lem}\label{thm:reg_exp_map} The expansion (\ref{eq:exp_taylor}) holds uniformly for all unit vectors $u\in T_pM$.
\end{lem}
\begin{proof}
The remainder term $R_u(s)$ can be expressed as
\[
		R_u(s) = \frac{1} {2} \int_0^{s}  \left(\frac{\partial } {\partial t} \Pi_{c_u(t)}(c_u'(t),c_u'(t)) \right) t^2 dt.
\]
Therefore, the claim follows from $c_u''(t)=\Pi_{c_u(t)}(c_u'(t),c_u'(t))$, $\left\|c_u'(t)\right\|=1$ and the Lipschitz continuity of $\Pi$.
\end{proof}
 
First we use (\ref{eq:exp_taylor}) in order to establish a metric curvature condition introduced in \cite{AB10} for the boundary of a class of convex bodies. A \emph{chord} of the boundary $\partial K$ of a convex body $K$ is a segment in $K$ that connects two points on $\partial K$. The \emph{base angle at $p$} of a chord $\gamma$ of $\partial K$ at an endpoint $p$ of $\gamma$ is the angle formed by the direction of $\gamma$ at $p$ and $T_pK$. According to \cite[Definition~4.1]{AB10} the boundary $\partial K$ has \emph{extrinsic curvature} $\geq A >0$ \emph{in the base-angle sense at $p$} if the base angles $\alpha$ at $p$ of chords of length $s$ from $p$ satisfy
\[
				\liminf_{s\To 0} 2\alpha/s \geq A.
\]
 If this holds for all boundary points, then $\partial K$ is said to have \emph{extrinsic curvature} $\geq A >0$ \emph{in the base-angle sense}.

\begin{lem}\label{lem:extrinsic_curvature} Let $K$ be a convex body in $\R^n$ whose boundary is of class $\mathcal{C}^{2,1}$ and has a positively definite second fundamental form. Then the boundary of $K$ has extrinsic curvature $\geq A$ for some $A>0$ in the base-angle sense.
\end{lem}
\begin{proof} For $p \in M=\partial K$ we represent points on $M$ close to $p$ as $\exp_p(su)$ with unit vectors $u\in T_pM$. From the Taylor expansion (\ref{eq:exp_taylor}) we obtain the following Taylor expansion for the distance $d$ from $\exp_p(su)$ to the tangent space $T_pM$
\[
		d = \frac 1 2 \left\| \Pi(u,u) \right\| s^2 +o(s^2).
\]
Moreover, for the extrinsic distance between $p$ and $\exp_p(su)$ we obtain $\left\|p- \exp_p(su)\right\| = s + o(s).$ Both together imply for the base angle $\alpha$ between $T_pM$ and the segment between $p$ and $\exp_p(su)$ that $\alpha = \frac 1 2 \left\| \Pi(u,u) \right\| s +o(s)$. Hence, $\partial K$ has extrinsic curvature $\geq \left\| \Pi(u,u) \right\| > 0$ at $p$. Since $\Pi$ is positive definite by assumption, the lemma follows now by compactness and continuity.
\end{proof}
Lemma \ref{lem:extrinsic_curvature} allows us to apply a result of Alexander and Bishop in the proof of the following statement, which generalizes results of Halpern \cite{Hal77} and Gruber \cite{Gr90}.

\begin{lem}\label{lem:convergence_to_boundary} Let $K$ be a convex body in $\R^n$ whose boundary is of class $\mathcal{C}^{2,1}$ and has a positively definite second fundamental form. If the initial directions of billiard trajectories converge to a tangent vector in the boundary, then the trajectories converge locally uniformly to the boundary.
\end{lem}
In particular, no quasigeodesic in the boundary can leave the boundary, and every billiard trajectory in the interior can be extended for all times without experiencing infinitely many bounces in finite time. 
\begin{proof} We need to show that a billiard trajectory restricted to any compact domain stays in arbitrarily small neighborhoods of the boundary, if its initial direction is sufficiently close to a tangent direction of the boundary. We call the angle between the tangent space of the boundary and the foward velocity at a bounce point the base angle. To prove what we need it is sufficient to observe that there exist constants $C,C',C''>0$ that only depend on $K$ such that the following holds:
\begin{compactenum}
\item If a billiard trajectory bounces with base angle at least $\alpha$, then the distance to the next bounce point is at least $C\alpha$.
\item If a billiard trajectory bounces with a sufficiently small base angle $\alpha$, then the next base angle is bounded from above by $\alpha+C' \alpha^2$.
\item If a billiard trajectory restricted to some compact interval bounces and all base angles are bounded by $\alpha$, then this trajectory stays in a $C'' \alpha^2$ neighborhood of the boundary of $K$.
\end{compactenum}
Indeed, by (ii) it requires at least $\frac{1}{4C'\alpha}$ bounces to increase the base angle from a sufficiently small base angle $\alpha$ to base angle $2\alpha$. After these bounces the trajectory has travelled at least a distance $\frac{C}{4C'}$ within a small neighborhood of the boundary. By choosing the initial base angle sufficiently small we can thus guarantee that the trajectory stays in an arbitrarily small neighborhood of the boundary for arbitrary long time.

To prove (i)-(iii) we apply a result by Alexander and Bishop on the convexity of the distance function to the boundary in Alexandrov spaces \cite{AB10}. More precisely, Lemma \ref{lem:extrinsic_curvature} allows us to apply Theorem~1.8 from \cite{AB10} which says that there exists some $R>0$ with $R>\mathrm{dist}_{\partial K} (p)$ for all $p \in K$ such that the function $f(p)= \left( R - \mathrm{dist}_{\partial K} (p)\right)^2$ satisfies $f'' \geq 2$ along geodesics in $K$. Note that for a sufficiently small neighborhood $U$ of ${\partial K}$ the distance function $\mathrm{dist}_{\partial K}$ is of class $\mathcal{C}^{2,1}$ on $U\cap K$, see (proof of) \cite[Lemma~14.16]{GT01}, and hence so is $f$. For a chord $\gamma :[a,b] \rightarrow K$ the base angles $\alpha$ at $\gamma(a)$ and $\beta$ at $\gamma(b)$ satisfy $(f\circ \gamma)'(a) = -R \sin(\alpha)$ and $(f\circ \gamma)'(b) = R \sin(\beta)$ as the differential of the normal exponential map at the zero section is the identity. At this point the claim is implied by the subsequent lemma since the $\mathcal{C}^{2,1}$-norms of $g:=f\circ \gamma$ can be bounded independently of $\gamma$.
\end{proof}

\begin{lem}\label{lem:convex_function} Let $g:[a,b] \rightarrow \R$ be a $\mathcal{C}^{2,1}$-function which satisfies $g(a)=R=g(b)$ for some $R>0$, $g(x)<R$ and $g''(x) \geq \delta$ for all $x\in(a,b)$ and some $\delta >0$. Then the following properties hold if $|g'(a)|$ is sufficiently small compared to $\delta$.
\begin{compactenum}
\item $C'|g'(a)| <|b-a|< C |g'(a)|$ for some $C,C'>0$. 
\item $|g'(a)+g'(b)| < C |g'(a)|^2$ for some $C>0$.
\item $|R-g(x)| < C|g'(a)|^2$ for some $C>0$ and all $x \in [a,b]$.
\end{compactenum}
Moreover, the constants $C$ and $C'$ only depend on $R$, $\delta$ and the $\mathcal{C}^{2,1}$-norm of $g$.
\end{lem}
\begin{proof} In the following $C$ and $C'$ always denote some constant that only depends on $R$, $\delta$ and the $\mathcal{C}^{2,1}$-norm of $g$. Since $g(x)$ has a minimum on $[a,b]$ and consequently $g'(x)$ a root, the lower bound in $(i)$ follows from Lipschitz continuity of $g'$. The first derivative $g'$ is strictly monotonically increasing on $[a,b]$, positive on $[a,b]\cap (a+|g'(a)|/\delta, \infty)$ and larger than $|g'(a)|$ on $[a,b]\cap (a+2|g'(a)|/\delta, \infty)$, because of $g'' \geq \delta$. The first two properties imply that $g$ is bounded from below by $R-|g'(a)|^2/\delta$, which proves $(iii)$. The lower bound on $g'(x)$ for $x>a+2|g'(a)|/\delta$ then proves the upper bound in $(i)$. 

To show $(ii)$ we compare $g$ with the degree $2$ polynomial $p$ determined by $p(a)=g(a)=R$, $p'(a)=g'(a)$ and $p''(a)=g''(a)$. Let $\hat b$ be the larger root of $p-R$. Because of $p'(a)=-p'(\hat b)$ we have 
\begin{equation} \label{eq:tri_split}
		|g'(a)+g'(b)|=|p'(\hat b)-g'(b)|\leq |p'(\hat b)-p'(b)| + |p'(b)-g'(b)|.
\end{equation}
For $x\in [a,b]$ we can estimate the difference $h=p-g$ as follows: $|h''(x)|\leq C (x-a)$ and so
\[
		|h'(x)| \leq \int_a^x |h''(x)| dx \leq C (x-a)^2, \, |h(x)| \leq \int_a^x |h'(x)| dx \leq C (x-a)^3.
\]
The bound on $|h'(b)|$ together with $(i)$ yields $|p'(b)-g'(b)|<C |g'(a)|^2$. The bound on $|h(b)|$ together with $(i)$ yields $|p(b)-R|< C |g'(a)|^3$. Solving a quadratic equation now implies that $|\hat b-b| < C|g'(a)|^2$ if $|g'(a)|$ is sufficiently small compared to $\delta$. Hence, $|p'(\hat b)-p'(b)|<C |g'(a)|^2$ by Lipschitz continuity of $p$. Combining these implications with (\ref{eq:tri_split}) finally shows $(iii)$.
\end{proof}
For the proof in the higher dimensional case we moreover need to show that limits of billiard trajectories in the boundary are quasigeodesics of the boundary.

\begin{lem}\label{lem:quasigeodesic_boundary} Let $K$ be a convex body in $\R^n$ whose boundary $M$ is of class $\mathcal{C}^{2,1}$. Then a quasigeodesic of $K$ which is contained in $M$ is also quasigeodesic of $M$ (with respect to its intrinsic metric).
\end{lem}
\begin{proof} Let $\gamma$ be a quasigeodesic of $K$ that is contained in $M$. By \cite[Proposition~2.3.12]{MR1835418} it is also parametrized by arclength with respect to the intrinsic metric of $M$. Now we want to apply the following characterization.

By \cite[Proposition~1.7]{QG95} a curve $\gamma:[a,b] \rightarrow X$ parametrized by arclength in an Alexandrov space $(X,d)$ (with curvature $\geq \kappa$) is a ($\kappa$-)quasigeodesic if and only if for every $t\in (a,b)$
\begin{equation}\label{eq:quasigeodesic_condition}
				\frac{1}{2}(d_q^2\circ \gamma)''(t) \leq 1+ o\left(d(q,\gamma(t))\right).
\end{equation}
Here a continuous function $\phi$ on $(a,b)$ is said to satisfy $\phi''\leq B$ if $\phi(t+\tau)\leq \phi(t)+ A \tau + B\tau^2/2+o(\tau^2)$ for some constant $A \in \R$ (that depends on $q$).

We denote the extrinsic and the intrinsic distance function on $M$ by $d$ and $d_i$, respectively. Locally around a point $p=\gamma(t)$ we write $\gamma$ as the image of a curve $\bar \gamma$ in $T_pM$ under the exponential map of $M$. More precisely, we choose the parametrization of $\bar \gamma$ such that $\gamma(t+\tau)=\exp_p(\bar \gamma (\tau))$. Moreover, we write a point $q\neq p$ close to $p$ as $q=\exp_p(su)$ for some unit vector $u\in T_pM$. By Lemma \ref{thm:reg_exp_map} we have Taylor expansions
\begin{equation*}
\begin{split}
\exp_p(su) &=p+su+\frac{1}{2} \Pi(u,u)s^2 + o(s^2), \\
\exp_p(\bar \gamma (\tau)) &=p+\bar \gamma (\tau)+\frac{1}{2} \Pi(\bar \gamma (\tau),\bar \gamma (\tau)) + o(\left\|\bar \gamma (\tau)\right\|^2). \\
\end{split}
\end{equation*}
Consider the orthogonal projection $P: M \To T_pM$. We define $\phi: T_pM \To T_pM$ as $\phi(v)=P(\exp_p(v)-p)$ and write $su=\phi(s'u')$ for some unit vector $u'\in T_pM$ and some $s'> 0$. On a sufficiently small neighborhood $V$ of the origin in $T_pM$ the map $\phi$ defines a homeomorphism onto its image. Moreover, the Taylor expansion above implies that 
\[
\frac 1 2 \left\|v\right\| \leq \left\|\phi(v)\right\| \leq 2 \left\|v\right\| 
\]
for all $v\in V$ if $V$ is sufficiently small. In other words, we have $\frac s 2 \leq s' \leq 2 s$ and so $o(s)=o(s')$.

Since $M$ is nonnegatively curved, Toponogov's theorem \cite{BGP92,MR1835418} implies
\[
  f(\tau):= \frac 1 2 d_i^2(\exp_p(su), \exp_p(\bar \gamma (\tau))) \leq \left\| su-\bar \gamma(\tau) \right\|^2.
\]
Moreover, $f(0)=\left\|u\right\|^2s^2$. In order to show that $\gamma$ satisfies condition (\ref{eq:quasigeodesic_condition}) with respect to the intrinsic metric of $M$ we need to estimate $f(\tau)-f(0)$ from above:
\begin{equation*}
\begin{split}
2(f(\tau)-f(0)) &\leq \left\| su-\bar \gamma(\tau) \right\|^2 - \left\|u\right\|^2s^2 \\
&=  - 2 \left\langle \phi(s'u'), \bar\gamma(\tau)\right\rangle + \left\|\bar\gamma(\tau) \right\|^2\\
&\leq -  2 \left\langle \exp_p(s'u')-p , \bar\gamma(\tau)\right\rangle +  \left\|\bar\gamma(\tau)+\frac 1 2 \Pi(\bar \gamma (\tau),\bar \gamma (\tau)) \right\|^2 \\
&\leq  - 2 \left\langle \exp_p(s'u')-p , \exp_p(\bar \gamma(\tau)) - p -\frac 1 2 \Pi(\bar \gamma (\tau),\bar \gamma (\tau)) \right\rangle  \\
	&  \quad + \left\|\exp_p(\bar \gamma(\tau)) - p \right\|^2 +o(\tau^2) \\
 &= \left\| (\exp_p(s'u')-p) - (\exp_p(\bar \gamma(\tau)) - p)  \right\|^2 - \left\| \exp_p(s'u')-p \right\|^2 \\
&  \quad  + 2 \left\langle \exp_p(s'u')-p- s'u' , \frac 1 2 \Pi(\bar \gamma (\tau),\bar \gamma (\tau)) \right\rangle+o(\tau^2) \\
&\leq d\left(\exp_p(\bar \gamma(\tau)),\exp_p(s'u')\right)^2 - d(p,\exp_p(s'u'))^2 + o(s')\tau^2 +   o(\tau^2),  \\
\end{split}
\end{equation*}
where we have used the Cauchy--Schwarz inequality several times to estimate terms. In the last step we have also used $\left\|\bar \gamma (\tau)\right\| \leq \tau$ so that $\left\|\Pi(\bar \gamma (\tau),\bar \gamma (\tau))\right\|<C \tau^2$ for some constant $C>0$.

Since (\ref{eq:quasigeodesic_condition}) is satisfied with respect to the extrinsic metric by assumption and since we have $o(d(p,\exp_p(s'u'))) =o(s')=o(s)= o(d_i(p,q))$, the above estimate shows that (\ref{eq:quasigeodesic_condition}) is also satisfied with respect to the intrinsic metric.\end{proof}

Now we can complete the proof of Theorem \ref{thm:convex_smooth}.

\begin{proof}[Proof of Theorem \ref{thm:convex_smooth}] By Lemma \ref{lem:convergence_initial_directions}, Lemma \ref{sub:proof_thm_B} and Lemma \ref{lem:quasigeodesic_boundary} there is a uniquely defined quasigeodesic flow on $K$ which is hence continuous by Lemma \ref{lem:unique_implies_continuity}. By Lemma \ref{cor:convergence_sub_alex} also the induced billiard evolution on $K$ is continuous with respect to the subspace topology of $TK$. The statement about the locally uniform convergence follows from the fact that our billiard trajectories are parametrized proportionally to arclength.
\end{proof}

The following construction provides an example of a convex body with a continuous quasigeodesic flow although the boundary is only $\mathcal{C}^1$.

\begin{exl}\label{exe:c1_example} Attach two circular arcs of different radii and close the boundary smoothly. Then no quasigeodesic in the boundary can leave the boundary and so the uniquely defined quasigeodesic flow is continuous by Lemma \ref{lem:unique_implies_continuity}.
\end{exl}

\subsection{Optimality of Theorem \ref{thm:convex_smooth} in dimension $2$} \label{sub:optimality}

The following example shows that the statement of Theorem \ref{thm:convex_smooth} may fail if the boundary is only three times differentiable and positively curved everywhere.

\begin{prp}\label{prp_noncontinuous_example} There exists a convex body $K \subset \R^2$ whose boundary is three times differentiable and positively curved everywhere, but which does not admit a continuous quasigeodesic flow resp. billiard evolution.
\end{prp}
\begin{proof} An example of a convex billiard table $K$ whose boundary is three times differentiable and positively curved everywhere with a unit speed billiard trajectory $c:[0,t_0]\To X$ whose bounce times accumulate at $t_0$ (for the first time) is constructed in \cite{Hal77}. The example constructed in \cite{Hal77} is such that the only possible extension of $\tilde c$ to a quasigeodesic $\tilde c:[0,\infty) \To X$ is the one which parametrizes the boundary after time $t_0$.

Suppose that the convex body $K$ from this construction admits a continuous quasigeodesic flow $\Phi: [0,\infty)\times TK \To TK$. Then the curve $[0,t_0] \ni t\mapsto \phi(t,c^+(0))$ projects to $c$, because the behaviour of $c$ on $[0,t_0]$ is uniquely determined by the initial condition $c^+(0)$. At time $t_0$ the velocity of $c$ converges to a tangent vector $v$ of the boundary of $K$ \cite[Theorem~1(a)]{Hal77}. Let $\gamma :\R \rightarrow K$ be a unit speed parametrization of the boundary of $K$ with $\gamma^+(0)=v$ and period $T$. The example constructed in \cite{Hal77} is such that $-v$ is the only direction in which two distinct quasigeodesics start. This is because the boundary has three bounded derivatives in the complement of $\gamma((-\varepsilon,0))$ for any $\varepsilon>0$, see the proof of Lemma \ref{lem:convergence_to_boundary} and of Theorem 3 in \cite{Hal77}. Therefore, the two sequences of initial directions $c^-(t_0-1-n)$ and $\gamma^-(T-1/n)$ have the same limit direction, but the quasigeodesics defined by these initial directions do not converge pointwise in contradiction to our continuity assumption.
\end{proof}

In the proof we used that for the convex body constructed in \cite{Hal77} there is only one tangent vector of the boundary in which two quasigeodesics $\gamma_i:[0,a) \To K$, $i=1,2$, start that are distinct on any nontrivial time interval $[0,\varepsilon)$. A negative answer to the following question would in particular imply that continuous quasigeodesic flows on convex bodies are unique, see Section \ref{sub:further_properties}.

\begin{qst} Does there exist a convex body with a continuous quasigeodesic flow on which a quasigeodesic can escape the boundary to the interior?
\end{qst}

\subsection{Continuous billiard flows on tables with non-smooth boundary}\label{sec:continuous_nonsmooth_bound}

Now we illustrate in dimension $2$ that the presence of continuous quasigeodesic flow on a convex body implies some local rigidity.

\begin{prp}\label{prp:local_rigidity_dim2} Let $K\subset \R^2$ be a convex body that admits a continuous quasigeodesic flow. Then each tangent cone of $K$ does so. In particular, all tangent cones are orbifolds. 
\end{prp}
\begin{proof} The second claim is immediate from Lemma \ref{lem:dim2_billiard}. Let $p$ be a point on the boundary of $K$. Let $l$ be a bisector of $T_pK$ through $p$ and let $\gamma^n_1$ and $\gamma^n_2$ be two families of quasigeodesics of $K$ such that the initial conditions of $\gamma^n_1$ and $\gamma^n_2$ are mirror images of each other with respect to $l$, $\gamma^n_1(0)$ and $\gamma^n_2(0)$ have distance $1/n$ to $l$ and distance $C/n$ to $p$ for some $n$-independent $C>0$, the initial directions of $\gamma^n_1$ and $\gamma^n_2$ are parallel to $l$, and $\gamma^n_1$ and $\gamma^n_2$ initially move towards $p$. Then by our continuity assumption $\gamma^n_1$ and $\gamma^n_2$ both converge to a common limit quasigeodesic $\gamma$ with $\gamma(0)=p$. 

Blowing up $K$ at $p$ with scaling factors $n$ gives a sequence of pointed metric spaces that converge to $T_pK$ with respect to the pointed Gromov-Hausdorff convergence. The sequences $\gamma^n_1$ and $\gamma^n_2$ converge to quasigeodesics $\gamma_1$ and $\gamma_2$ on $T_pK$, cf. \cite[Section~5.1.6]{Pe07}. By construction $\gamma_1$ and $\gamma_2$ are mirror images of each other with respect to $l$, initially they both move parallel to $l$ towards $p$ and eventually, after finitely many bounces, both become parallel to $\gamma$. Now the proof of Lemma \ref{lem:dim2_billiard} and scale invariance of $T_pK$ implies that the opening angle of $T_pK$ at $p$ is of the form $\pi/m$ for some natural number $n$, i.e. that $T_pK$ is an orbifold.
\end{proof}

We remark that considering the curvature of the boundary shows that there can be at most four non-smooth orbifold points in the boundary and that the table is a rectangle if there are four, see e.g. \cite{AP18}. Moreover, such tables have continuous quasigeodesic flows if the boundary is sufficiently regular otherwise.

We expect that the statements of Proposition \ref{prp:local_rigidity_dim2} are also true in higher dimensions, see Conjecture \ref{conj:orb_point}.

\section{Boundaries of polyhedral convex bodies} \label{sec:boundary_polyhedral}

In this section we prove Theorem \ref{thm_3_4_polyhedral} which relies on the notion of a (continuous) quasi\-geodesic flow. However, to read it one can either take the next paragraph for granted or recall the definition of a quasigeodesic flow from Section \ref{sec:continuous_billiard_quasigeodesic} and the characterization of quasigeodesics in polyhedral Alexandrov spaces provided in Lemma \ref{cor:quasigeodesic_polyhedral}: A quasigeodesic is a curve parametrized by arclength which is locally length realizing except at a discrete set of times at which the forward and backward derivatives are polar. As far as continuity concerns one can alternatively demand it only at directions over points that are contained in the interior of a maximal dimensional simplex where continuity can be defined in terms of the topology of the ambient Euclidean vector space.

The only if statement of the first part of Theorem \ref{thm_3_4_polyhedral} works analogously by induction on the dimension as in the proof of Theorem \ref{thm:billiard_orbifold} in Section \ref{sec:polyhedral_billiard_table}. The statement that the orbifold geodesic flow defines a continuous quasigeodesic flow follows as in the proof of Proposition \ref{prp:polyhedral_continuous} and was also observed in general in Section \ref{sec:continuous_billiard_quasigeodesic}.

\subsection{Proof of Theorem \ref{thm_3_4_polyhedral} in dimension $3$} \label{sub:_prove_C_dim3}

For the proof of the second part of Theorem \ref{thm_3_4_polyhedral} we first deal with the $3$-dimensional statement.

\begin{prp}\label{prp:3d_tetrahedron} Suppose a polyhedral convex body $K$ in $\R^3$ is bounded by an orbifold. Then $K$ is a simplex with all four cone angles equal to $\pi$.
\end{prp}

We present two different proofs of Proposition \ref{prp:3d_tetrahedron}, one using the Gauß--Bonnet theorem and another one via Euler's polyhedral formula.

\begin{proof}[Proof of Proposition \ref{prp:3d_tetrahedron} via Gauß--Bonnet]
We assume that the reader is familiar with the curvature measure and the Gauß--Bonnet theorem for $2$-dimensional convex surfaces, see e.g. \cite{AP18,Al05}. The curvature of the boundary of a $3$-dimensional polyhedral convex body is concentrated in the vertices of the body. The curvature at a vertex equals the area of the intersection of the normal cone and a unit sphere at this vertex. However, it is determined by the intrinsic geometry of the surface alone. By the Gauß--Bonnet theorem applied to a neighborhood of the vertex, the curvature $\kappa$ and the cone angle $\alpha$ at vertex $i$ are related via $\kappa=2\pi-\alpha$. Here the cone angle at a vertex $p$ is the length of the boundary of the intersection of $T_pK$ with the unit sphere, and it equals the sum of the face angles adjacent to $p$.

We enumerate the vertices of $K$ from $1$ to $k$ and denote the curvature and the cone angle at the $i$-th vertex by $\kappa_i$ and $\alpha_i$, respectively. Applying Gauß--Bonnet to the entire surface yields that
\[
				4\pi=\sum_{i=1}^k \kappa_i=2k\pi - \sum_{i=1}^k \alpha_i.
\]
By our orbifold assumption each $\alpha_i$ is of the form $\alpha_i=\frac{2\pi}{n_i}$ for some $n_i\in \N$. Since the normal cone at a vertex has nonempty interior, the curvature at a vertex is positive and the cone angle is strictly less than $2\pi$. This implies that $n_i \geq 2$ for all $i=1\ldots, k$. Hence, we obtain the same condition that we encountered in the proof of Lemma \ref{lem:dim2_billiard}, namely
\[
			\frac{1}{n_1}+ \ldots + \frac{1}{n_k} = k-2
\]
with $n_i\in \N_{\geq2}$ for all $i=1,\ldots,k$. Since a convex body in $\R^3$ has at least four vertices, this time the only possible solution is $k=4$ and $n_1=n_2=n_3=n_4=2$. It corresponds to a simplex with all four cone angles equal to $\pi$.
\end{proof}

The formulation in Proposition \ref{prp:3d_tetrahedron} is related to the formulation in Theorem \ref{thm_3_4_polyhedral} via the following observation, cf. \cite[Section~4]{AP18} and Figure \ref{fig:_triangle}.

\begin{exe}  All cone angles of a $3$-simplex are $\pi$ if and only if opposite sides have equal length. For each acute triangle $T$ in the plane there exists precisely one such $3$-simplex all of whose faces are congruent to $T$, cf. Figure \ref{fig:_triangle}.
\end{exe}

For the alternative proof of Proposition \ref{prp:3d_tetrahedron} without the Gauß-Bonnet theorem we need the following two ingredients.

\begin{lem}\label{lem:cone_angle_bound} The cone angle at a vertex of a polyhedral convex body $K$ in $\R^3$ is less then $2\pi$.
\end{lem}
\begin{proof} Recall that the cone angle at a vertex $p$ is the length of the boundary of the intersection $P$ of $T_pK$ with the unit sphere. This intersection $P$ can be seen as a finite intersection of at least $3$ hemispheres in the unit sphere. In this intersection it suffices to consider hemispheres that correspond to sides of the spherical polygon $P$. By the spherical triangle inequality the length of the boundary of this intersection strictly increases if we subsequently remove hemispheres from the intersection until only two hemispheres are left. At this final stage the length of the boundary is $2\pi$. Since the original number of hemispheres was at least $3$, the claim follows.
\end{proof}

Let $T$ be a triangulation of a disk. We denote the number of vertices, edges and faces of such a triangulation by $V$, $E$ and $F$, respectively. The following second ingredient can for instance be easily obtained by induction on the number of faces.
\begin{exe} \label{lem:disk_inequality} Let $T$ be a triangulation of a disk. Then $2V \leq E + 3$, or equivalently $V \leq F + 2$ by Euler's formula.
\end{exe}

Now we present the second proof of Proposition \ref{prp:3d_tetrahedron}.

\begin{proof}[Proof of Proposition \ref{prp:3d_tetrahedron} via Euler's formula] We can triangulate the boundary of $K$ in such a way that every vertex of the triangulation is also a vertex of $K$. All interior angles of the triangulation sum up to $F\pi$. On the other hand, by Lemma \ref{lem:cone_angle_bound} and our orbifold assumption the cone angle at a vertex is at most $\pi$. Hence, we have that $F \leq V$. With Euler's polyhedral formula $V-E+F=2$ we deduce that
\begin{equation}
				E \leq 2V-2.
\end{equation}
We claim that we actually have equality. To see this we can pick a vertex and remove its star from $T$ to obtain a new triangulation $T'$ of a disk, which satisfies $2V' \leq E' + 3$ by Exercise \ref{lem:disk_inequality}. Because of $V'=V-1$ and $E'\leq E-3$ we indeed have equality. Since we can start this argument with any vertex in $T$, the triangulation $T$ must be trivalent. This implies that there are only four vertices. The equality discussion moreover shows that the total angle at each vertex is $\pi$.
\end{proof}

\begin{figure}
	\centering
		\def\svgwidth{0.3\textwidth}
		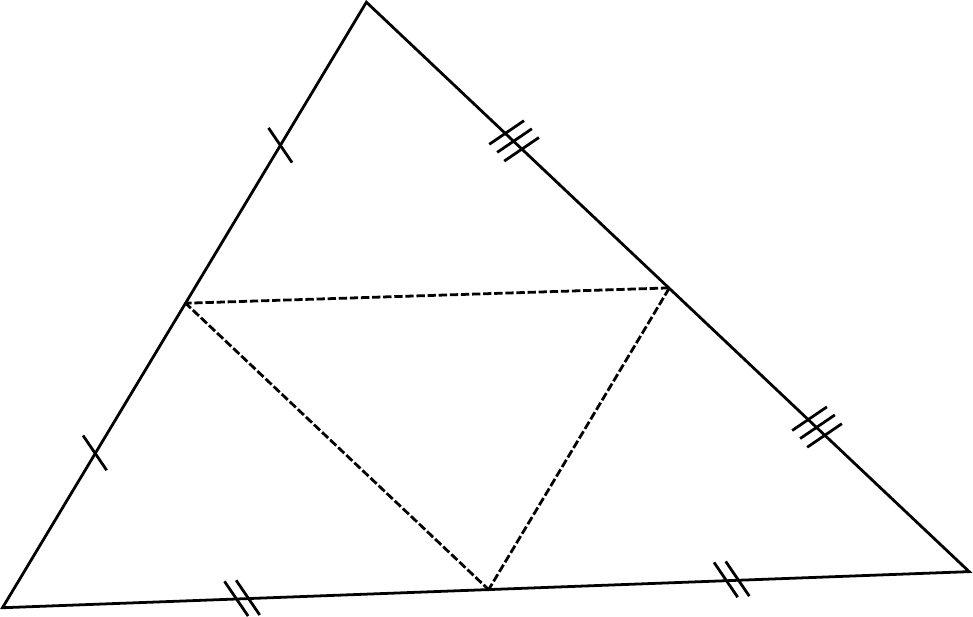
	\caption{Acute triangle subdivided into four similar triangles by segments between sidemidpoints.}
	\label{fig:_triangle}
\end{figure}

\subsection{Proof of Theorem \ref{thm_3_4_polyhedral} in higher dimensions} \label{sub:proof_C_dimh}

In order to rule out examples in higher dimensions we will apply the following intrinsic characterization of faces. In fact, we will only need the characterization of faces of codimension $3$, which we have already used, but for completeness we state the result for all codimensions.

\begin{prp} \label{prp:face_characterization} Let $K$ be a polyhedral convex body in $\R^n$, $n\geq3$. Then a point in the boundary of $K$ belongs to the interior of a face of dimension $l<n-1$ if and only if a neighborhood of $p$ in $\partial K$ (with its induced intrinsic metric) isometrically splits off an open set in $\R^{l}$ but not an open set in $\R^{l+1}$.
\end{prp}
\begin{proof} For basic properties about polyhedral convex bodies used in the following proof we refer the reader to \cite{Gr03}, in particular Section~3.1. 

We only need to show that no neighborhood of a point $p$ in a face of dimension $l<n-1$ splits off an open set in $\R^{l+1}$. Looking at the intersection of $K$ with a $(n-l)$-dimensional plane through $p$ and orthogonal to the supporting face of $p$ reduces the claim to the case $l=0$. We prove the latter by induction on $n$. 

For $n=3$ the cone angle at each vertex is strictly less than $2 \pi$ by Lemma \ref{lem:cone_angle_bound} or by the Gauß-Bonnet, cf. proof of Proposition \ref{prp:3d_tetrahedron}. In this case no shortest curve can pass through a vertex \cite[1.8.1~(A)]{Al05}. This proves the claim for $n=3$.

To prove the claim for some $n>3$, we first observe that in this case there are at least $n$ edges, i.e. $1$-dimensional faces, adjacent to a vertex $p$. For, a hyperplane $H$ that cuts off a small neighborhood of $p$ intersects $K$ in an $(n-1)$-dimensional polytope whose vertices are precisely the intersections of $H$ with the edges adjacent to $p$, and a polytope is a convex hull of its vertices \cite[Section~2.4, Theorem~2.3.4]{Gr03}. By induction assumption no point on these edges admits a neighborhood that splits off an open set in $\R^2$. Now suppose that a neighborhood of $p$ splits off an open set in $\R$. Then there exists an edge adjacent to $p$ that is not contained in this $\R$-factor. Therefore, points on this edge have neighborhoods that split off open sets in $\R$ in two different directions. Now a splitting theorem of Milka \cite{Mi67} (applied to a tangent cone which is locally isometric to a neighborhood of the base point) implies that neighborhoods of such points actually split off an open set in $\R^2$. This contradiction completes the proof of the proposition.
\end{proof}

\begin{prp} \label{prp:polyhedral_orbifold_4d} For $n\geq 4$ there does not exist a polyhedral convex body in $\R^n$ whose boundary is an orbifold.
\end{prp}
\begin{proof} Suppose such a polyhedral convex body $K$ exists in some $\R^n$, $n\geq 4$. We pick a point $p$ in the interior of a face of codimension $4$ and look at a small neighborhood of $p$ in the intersection of the boundary of $K$ with a $4$-dimensional plane through $p$ orthogonal to the supporting face of $p$. This intersection is a $3$-dimensional orbifold and by Proposition \ref{prp:face_characterization} the faces of codimension $3$ belong to the codimension $2$ stratum of this orbifold. On a polyhedral convex body each codimension $4$ face is adjacent to at least $4$ codimension $3$ faces, see proof of Proposition~\ref{prp:face_characterization}. However, on the other hand, in a $3$-orbifold at most $3$ components of the codimension $2$ stratum can meet at a point. The latter follows from the classification of finite subgroups of $\Or(3)$. This contradiction completes the proof of the proposition and of Theorem \ref{thm_3_4_polyhedral}.
\end{proof}

\end{document}